\documentclass[a4paper,12pt]{article}
\usepackage{a4wide}
\usepackage{amsmath}
\usepackage{amssymb}
\usepackage{amsthm}
\usepackage{latexsym}
\usepackage{graphicx}
\usepackage[english]{babel}
\usepackage{makeidx}
\usepackage[mathlines]{lineno}
\newcommand{\stir}{\genfrac{[}{]}{0pt}{}}

\newtheorem{obs} [subsection]{Remark}

\newtheorem{prop}[subsection]{Proposition}

\newtheorem{teor}[subsection]{Theorem}
\newtheorem{lema}[subsection]{Lemma}
\newtheorem{cor} [subsection]{Corollary}

\newcommand{\paa}{p_{\mathbf a}}
\newcommand{\pak}{p_{k,\mathbf a}}
\newcommand{\Pak}{P_{k,\mathbf a}}
\newcommand{\Pa}{P_{\mathbf a}}
\newcommand{\za}{\zeta_{\mathbf a}}
\newcommand{\zak}{\zeta_{k,\mathbf a}}

\def\p{\operatorname{p}}
\def\pp{\operatorname{pp}}

\begin{document}
\selectlanguage{english}
\frenchspacing

\numberwithin{equation}{section}

\title{On the restricted $k$-multipartition function}
\author{Mircea Cimpoea\c s$^1$ and Alexandra Teodor$^2$}
\date{}

\maketitle

\footnotetext[1]{ \emph{Mircea Cimpoea\c s}, University Politehnica of Bucharest, Faculty of
Applied Sciences, 
Bucharest, 060042, Romania and Simion Stoilow Institute of Mathematics, Research unit 5, P.O.Box 1-764,
Bucharest 014700, Romania, E-mail: mircea.cimpoeas@upb.ro,\;mircea.cimpoeas@imar.ro}
\footnotetext[2]{ \emph{Alexandra Teodor}, University Politehnica of Bucharest, Faculty of
Applied Sciences, 
Bucharest, 060042, E-mail: alexandra.teodor@upb.ro}

\begin{abstract}
Let $\mathbf a=(a_1,\ldots,a_r)$ be a sequence of positive integers and $k\geq 2$ an integer.
We study $\pak(n)$, the restricted $k$-multipartition function associated to $\mathbf a$ and $k$. We prove new
formulas for $\pak(n)$, its waves $W_j(n,k,\mathbf a)$'s and its polynomial part $\Pak(n)$. Also, we give a lower bound for the density of
the set $\{n\geq 0\;:\;\pak(n)\not\equiv 0(\bmod\;m)\}$, where $m\geq 2$ is an integer.

\textbf{Keywords}: Integer partition, Restricted partition function, Multipartition.

\textbf{MSC2010}: 11P81, 11P83.
\end{abstract}

\maketitle
\section{Introduction}

Let $n$ be a positive integer. We denote $[n]=\{1,2,\ldots,n\}$.
A partition of $n$ is a non-increasing sequence $\lambda=(\lambda_1,\ldots,\lambda_m)$ of positive integers such that 
$|\lambda|=\lambda_1+\cdots+\lambda_m=n$. We define
$\p(n)$ as the number of partitions of $n$ and for convenience, we define $p(0) = 1$.
This notion has the following generalization: 

Let $k\geq 2$ be an integer.
A $k$-component multipartition of $n$ is a $k$-tuple $\lambda=(\lambda^1,\ldots,\lambda^k)$ of partitions of $n$
such that $|\lambda|=|\lambda^1|+\cdots+|\lambda^k|=n$; see \cite{andrews2}. We denote $p_k(n)$, the number or
$k$-component multipartitions of $n$ and $p_k(0)=1$.

Let $\mathbf a := (a_1, a_2, \ldots , a_r)$ be a sequence of positive integers, $r \geq 1$. The \emph{restricted partition
function} associated to $\mathbf a$ is $\paa : \mathbb N \to \mathbb N$, $\paa(n) :=$ the number of integer solutions $(x_1, \ldots, x_r)$
of $\sum_{i=1}^r a_ix_i = n$ with $x_i \geq 0$. Note that the generating function of $\paa(n)$ is
\begin{equation}\label{gen}
\sum_{n=0}^{\infty}\paa(n)z^n= \frac{1}{(1-z^{a_1})\cdots(1-z^{a_r})},\;|z|<1.
\end{equation}

The \emph{restricted $k$-multipartion function} associated to $\mathbf a$ is $\pak(n):\mathbb N \to \mathbb N$, 
$\pak(n):=$ the number of vector solutions
$(x^1,\ldots,x^k)$ of $$\sum_{j=1}^k\sum_{i=1}^r a_ix^j_i=n,
\text{ where }x^j=(x^j_1,\ldots,x^j_r)\in \mathbb N^r\text{ for }1\leq j\leq k.$$

The aim of the paper is to study the properties of the function $\pak(n)$, following the methods used in our previous paper \cite{bul}.

We consider the sequence
\begin{equation}\label{ak}
\mathbf a[k]:=(a_1^{[k]},a_2^{[k]},\ldots,a_r^{[k]}),
\end{equation}
where $\ell^{[k]}$ denotes $k$ copies of $\ell$.

It is easy to see that $\pak(n)=p_{\mathbf a[k]}(n)$ and therefore, from \eqref{gen} and \eqref{ak} we have
\begin{equation}\label{genk}
\sum_{n=0}^{\infty}\pak(n)z^n= \frac{1}{(1-z^{a_1})^k\cdots(1-z^{a_r})^k},\;|z|<1.
\end{equation}
In Proposition \ref{p1} we show that
$$\zak(s,w_1,\ldots,w_k):= \prod_{i=1}^k \za(s,w_i) = \sum_{n=0}^{\infty} \sum_{n_1+\cdots+n_k=n} \frac{\pak(n)}{(n_1+w_1)^s\cdots (n_k+w_k)^s},$$
where $\za(s,w)$ is the Barnes zeta function (see \cite{barnes}). 

In Proposition \ref{p2} we express $\zak(s,w_1,\ldots,w_k)$ in terms of
Hurwitz zeta functions.

Let $D$ be the least common multiple of $a_1,\ldots,a_r$. In Proposition \ref{p3} we note that 
$$\pak(n)=d_{k,\mathbf a,rk-1}(n)n^{rk-1} + \cdots + d_{k,\mathbf a,1}(n)n + d_{k,\mathbf a,0}(n),$$
is a quasi-polynomial of period $D$. From this result, we deduce a new expression for $\zak(s,w_1,\ldots,w_k)$ in Corollary \ref{26}.

In Theorem \ref{teo1} we prove formulas for the periodic functions $d_{k,\mathbf a,m}(n)$. 

Using the fact that $\pak(n)=p_{\mathbf a[k]}(n)$, in Theorem \ref{teo2} we prove a formula for $\pak(n)$. In Proposition \ref{dety} we
show that if a certain determinant is nonzero, then $\pak(n)$ can be expressed in terms of values of Bernoulli polynomials and 
Bernoulli-Barnes numbers. Using a result from \cite{graj}, in Corollary \ref{grajd} we show that
$$ \lim_{N\to\infty} \frac{\#\{n\leq N\;:\;\pak(n)\not\equiv 0(\bmod\; m)\}}{N} \geq \frac{1}{k \sum\limits_{i=1}^r a_i},$$
where $m>1$ is an integer.

Similarly to $\paa(n)$ we consider the \emph{Sylvester decomposition} (see \cite{sylvester}, \cite{sylvesterc} and \cite{sylv}) of $\pak(n)$ as a sums of "waves", i.e. 
$$\pak(n)=\sum_{j\geq 1}W_j(k,\mathbf a,n),$$ 
where $W_j(k,\mathbf a,n):=W_j(\mathbf a[k],n)$. In Theorem \ref{teo3} we prove a formula for $W_j(k,\mathbf a,n)$.

The polynomial part of $\pak(n)$ is $\Pak(n):=W_1(k,\mathbf a,n)$. 

In Theorem \ref{teo4} and Theorem \ref{teo5} we prove new formulas for $\Pak(n)$.

\section{Preliminary results}

Let $r\geq 1$ and $k\geq 2$ be two integers. Let $\mathbf a=(a_1,\ldots,a_r)$ be a sequence of positive integers.
Let $D$ be the least common multiple of $a_1,\ldots,a_r$. 

For $0\leq j_1\leq \frac{D}{a_1}-1$, $0\leq j_2 \leq \frac{D}{a_2}-1, \ldots, 0\leq j_r \leq \frac{D}{a_r}-1$ let, 
by Euclidean division, $\mathfrak q(j_1,\ldots,j_r)$ and $\mathfrak r(j_1,\ldots,j_r)$ be the unique integers such that
\begin{equation}\label{21}
  a_1j_1+\cdots+a_rj_r=\mathfrak q(j_1,\ldots,j_r)D + \mathfrak r(j_1,\ldots,j_r),\;\; 0\leq \mathfrak r(j_1,\ldots,j_r) \leq D-1.
\end{equation}
We denote the rising factorial by $x^{(r)}:=(x+1)(x+2)\cdots (x+r-1)$, $x^{(0)}=1$. It holds that 
\begin{equation}\label{22}
\binom{n+r-1}{r-1} = \frac{1}{(r-1)!}n^{(r)} = \frac{1}{(r-1)!} \left(\stir{r}{r} n^{r-1}+\cdots +\stir{r}{2} n + \stir{r}{1}\right),
\end{equation}
where $\stir{r}{k}$'s are the \emph{unsigned Stirling numbers} of the first kind.

Let $w>0$ be a real number. The \emph{Barnes zeta function} associated to $\mathbf a$ and $w$ is
\begin{equation}\label{23}
\za(s,w):=\sum_{u_1,\ldots,u_r\geq 0}\frac{1}{(a_1u_1+\cdots+a_ru_r+w)^s},\; Re(s)>r.
\end{equation}
For basic properties of the Barnes zeta function see \cite{barnes}, \cite{rui} and \cite{spreafico}.

Let $w_1,\ldots,w_k>0$ be some real numbers. We consider the function
\begin{equation}\label{24}
\zak(s,w_1,\ldots,w_k):=\za(s,w_1)\cdots \za(s,w_k).
\end{equation}

\begin{prop}\label{p1}
We have that
$$\zak(s,w_1,\ldots,w_k) = \sum_{n=0}^{\infty} \sum_{n_1+\cdots+n_k=n} \frac{\pak(n)}{(n_1+w_1)^s\cdots (n_k+w_k)^s}.$$
\end{prop}

\begin{proof}
For $1\leq j\leq k$, we have that
$$\za(s,w_j)=\sum_{n_j=0}^{\infty}\frac{\paa(n_j)}{(n_j+w_j)^s}$$
Therefore, the conclusion follows from the definitions of $\pak(n)$ and $\zak(s,w_1,\ldots,w_k)$.
\end{proof}

\begin{obs}\rm
Note that, if $r=1$ and $1\leq j\leq k$ then 
$$\za(s,w_j)=\sum_{u_j=0}^{\infty}\frac{1}{(a_1u_j+w_j)^s}=\frac{1}{a_1^s}\sum_{u_j=0}^{\infty}\frac{1}{(u_j+\frac{w_j}{a_1})^s}
= \frac{1}{a_1^s}\zeta\left(s,\frac{a_1}{w_j}\right),$$
where $$\zeta(s,w):=\sum_{n=0}^{\infty}\frac{1}{(n+w)^s}, Re(s)>1$$ is the Hurwitz zeta function.
It follows that
$$\zak(s,w_1,\ldots,w_k)=\frac{1}{a_1^{ks}}\zeta\left(s,\frac{a_1}{w_1}\right)\cdots \zeta\left(s,\frac{a_k}{w_j}\right).$$
\end{obs}

We consider the set $$\mathbf B:=\{(j_1,\ldots,j_r)\;:\; 1\leq j_1\leq \frac{D}{a_1}-1,\ldots, 1\leq j_r \leq \frac{D}{a_r}-1 \}.$$
We recall the following result from \cite{lucrare}. Also, we mention that the definition of Stirling numbers is slightly 
different there; see \cite{cori} for more details.

\begin{lema}(\cite[Lemma 2.2]{lucrare})\label{lem22} We have 
$$\za(s,w)=\frac{1}{D^s(r-1)!} \sum_{(j_1,\ldots,j_r)\in\mathbf B} \sum_{k=0}^{r-1} \stir{r}{k+1}
\sum_{j=0}^k  (-1)^j\binom{k}{j}\left(\frac{a_1j_1+\cdots+a_rj_r+w}{D}\right)^j \times $$ 
$$\times \zeta(s-k+j,\frac{\mathfrak r(j_1,\ldots,j_r)+w}{D}).$$
\end{lema}

From \eqref{24} and Lemma \ref{lem22} it follows that:

\begin{prop}\label{p2}
We have
$$\zak(s,w_1,\ldots,w_k)=\frac{1}{D^s(r-1)!}  \sum_{m=0}^{r-1} \stir{r}{m+1} \sum_{\ell=0}^m (-1)^{\ell} \binom{m}{\ell} \times$$
 $$\times \prod_{i=1}^k \sum_{(j^i_1,\ldots,j^i_r)\in \mathbf B} \left(\frac{a_1j^i_1+\cdots+a_rj^i_r+w_i}{D}\right)^{\ell} 
 \zeta(s-m+\ell,\frac{\mathfrak r(j^i_1,\ldots,j^i_r)+w_i}{D}). $$
\end{prop}

\begin{prop}\label{p3}
$\pak(n)$ is a quasi-polynomial of degree $rk-1$, with the period $D$, i.e.
$$\pak(n)=d_{k,\mathbf a,rk-1}(n)n^{rk-1} + \cdots + d_{k,\mathbf a,1}(n)n + d_{k,\mathbf a,0}(n),$$
where $d_{k,\mathbf a,m}(n+D)=d_{k,\mathbf a,m}(n)$ for $0\leq m\leq rk-1$ and $n\geq 0$, and $d_{k,\mathbf a,rk-1}(n)$
is not identically zero.
\end{prop}

\begin{proof}
Since $\pak(n)=p_{\mathbf a[k]}(n)$, where $\mathbf a[k]=(a_1^{[k]},\ldots,a_r^{[k]})$ (see \eqref{ak}), the 
conclusion follows from the classical result of Bell \cite{bell}.
\end{proof}

\begin{cor}\label{26}
We have that
$$\zak(s,w_1,\ldots,w_k) = \sum_{n=0}^{\infty} \sum_{m=0}^{rk-1} d_{k,\mathbf a,m}(n) \sum_{\substack{n_1+\cdots+n_k=n \\ \ell_1+\cdots+\ell_k = m}} \binom{m}{\ell_1,\ldots,\ell_k} \prod_{j=1}^k \frac{1}{(n_j+w_j)^{s-\ell_j}(1+\frac{w_j}{n_j})^{\ell_j}}.$$
\end{cor}

\begin{proof}
From Proposition \ref{p1} we have that 
$$  \zak(s,w_1,\ldots,w_k) = \sum_{n=0}^{\infty} \sum_{n_1+\cdots+n_k=n} \frac{\pak(n)}{(n_1+w_1)^s\cdots (n_k+w_k)^s}.$$
Therefore, from Proposition \ref{p3} it follows that
\begin{align*}
& \zak(s,w_1,\ldots,w_k) = \sum_{n=0}^{\infty} \sum_{m=0}^{rk-1} \sum_{n_1+\cdots+n_k=n} \frac{d_{k,\mathbf a,m}(n)(n_1+\cdots+n_k)^{m}}{(n_1+w_1)^s\cdots (n_k+w_k)^s} = \\
& = \sum_{n=0}^{\infty} \sum_{m=0}^{rk-1} d_{k,\mathbf a,m}(n) \sum_{\substack{n_1+\cdots+n_k=n \\ \ell_1+\cdots+\ell_k = m}}
    \binom{m}{\ell_1,\ldots,\ell_k}\frac{n_1^{\ell_1}\cdots n_{k}^{\ell_k}}{(n_1+w_1)^s\cdots (n_k+w_k)^s}.
\end{align*}
The conclusion follows immediately.
\end{proof}


We fix two integers $N\geq 1$ and we consider the numbers
\begin{equation}\label{fsn1}
f_{N,\ell}=\#\{(i_1,\ldots,i_k)\;:\;i_1+\cdots+i_k=\ell,\;0\leq i_t\leq N-1\}\text{ where }0\leq\ell\leq k(N-1).
\end{equation}
It is clear that $f_{N,\ell}$ is the coefficient of $t^{\ell}$ of the polynomial
\begin{equation}\label{fsn2}
f_{N}(t)=(1+t+\cdots+t^{N-1})^k.
\end{equation}
Using the binomial expansion, we have
\begin{equation}\label{fsn3}
f_{N}(t)=(1-t^N)^k(1-t)^{-k}=\sum_{i=0}^k(-1)^i\binom{k}{i}t^{iN} \sum_{j=0}^{\infty}\binom{j+k-1}{j}t^j.
\end{equation}

\begin{prop}\label{snl}
With the above notations, we have that
$$f_{N,\ell}=\sum_{i,j\geq 0,\;iN+j=\ell}(-1)^i\binom{k}{i}\binom{j+k-1}{j}.$$
\end{prop}

\begin{proof}
The conclusion follows from \eqref{fsn1}, \eqref{fsn2} and \eqref{fsn3}.
\end{proof}

\section{Main results}

We use the notations from the previous section.

\begin{teor}\label{teo1}
For $n\geq 0$ we have that
\begin{align*}
& d_{k,\mathbf a,m}(n) = \frac{1}{(rk-1)!} \sum_{\substack{(\ell_1,\ldots,\ell_r)\in\mathbf C \\ a_1\ell_1+\cdots+a_r\ell_r \equiv n(\bmod\;D)}} 
\prod_{s=1}^r \sum_{i_s,j_s\geq 0,\;i_s\frac{D}{a_s}+j_s=\ell_s}(-1)^{i_s}\binom{k}{i_s}\binom{j_s+k-1}{j_s} \times
 \\
& \times \sum_{t=m}^{rk-1} \stir{rk}{t+1} (-1)^{t-m} \binom{k}{m} D^{-t}(a_1\ell_1+\cdots+a_r\ell_r)^{t-m},
\end{align*}
where $\mathbf C=\{(\ell_1,\ldots,\ell_r)\;:\;0\leq \ell_1 \leq k(\frac{D}{a_1}-1),\ldots,0\leq \ell_r\leq k(\frac{D}{a_r}-1)\}$.
\end{teor}

\begin{proof}
We consider the set
\begin{align*}
& \mathbf{B[k]}:=\{(j_1,\ldots,j_{rk})\;:\;0\leq j_1\leq \frac{D}{a_1}-1,\ldots, 0\leq j_k\leq \frac{D}{a_1}-1,\ldots \\
& \ldots,0\leq j_{rk-k+1}\leq \frac{D}{a_r}-1,\ldots 0\leq j_{rk}\leq \frac{D}{a_{r}}-1\}.
\end{align*}
According to \cite[Theorem 2.8]{lucrare} and Proposition \ref{p3} we have that
$$d_{k,\mathbf a,m} = \frac{1}{(rk-1)!} \sum_{\substack{(j_1,\ldots,j_{rk}) \in \mathbf{B[k]} \\ 
  a_1(j_1+\cdots+j_k)+\cdots+a_r(j_{rk-k+1}+\cdots+j_{rk}) \equiv n(\bmod\;D)}} \sum_{t=m}^{rk-1} \stir{rk}{t+1} (-1)^{t-m} \binom{k}{m} \times $$	
\begin{equation}\label{kuku}
\times D^{-t} (a_1(j_1+\cdots+j_k)+\cdots+a_r(j_{rk-k+1}+\cdots+j_{rk}))^{t-m}.
\end{equation}
We let $\ell_1:=j_1+\cdots+j_k$, $\ell_2:=j_{k+1}+\cdots+j_{2k},\ldots,\ell_r=j_{rk-k+1}+\cdots+j_{rk}$. It is clear
that $(j_1,\ldots,j_{rk})\in\mathbf{B[k]}$ implies $(\ell_1,\ldots,\ell_r)\in\mathbf C$. Since 
$$f_{\frac{D}{a_s},\ell_s}=\#\{(j_{ks-k+1},\ldots,j_{ks})\;:\;j_{ks-k+1}+\cdots+j_{ks}=\ell_s,\;0\leq j_t\leq \frac{D}{a_s}\text{ for }ks-k+1\leq t\leq ks\},$$
the conclusion follows from \eqref{kuku} and Proposition \ref{snl}.
\end{proof}

\begin{teor}\label{teo2}
For $n\geq 0$ we have that
\begin{align*}
& \pak(n) = \frac{1}{(rk-1)!} \sum_{\substack{(\ell_1,\ldots,\ell_r)\in\mathbf C \\ a_1\ell_1+\cdots+a_r\ell_r \equiv n(\bmod\;D)}} 
\prod_{s=1}^r \sum_{i_s,j_s\geq 0,\;i_s\frac{D}{a_s}+j_s=\ell_s}(-1)^{i_s}\binom{k}{i_s}\binom{j_s+k-1}{j_s} \times \\
& \times \prod_{t=1}^{rk-1}\left( \frac{n-a_1\ell_1-\cdots-a_r\ell_r}{D} + t \right),
\end{align*}
where $\mathbf C=\{(\ell_1,\ldots,\ell_r)\;:\;0\leq \ell_1 \leq k(\frac{D}{a_1}-1),\ldots,0\leq \ell_r\leq k(\frac{D}{a_r}-1)\}$.
\end{teor}

\begin{proof}
The proof is similar to the proof of Theorem \ref{teo1}, using \cite[Corollary 2.10]{lucrare}.
\end{proof}

The \emph{Bernoulli polynomials} are defined by
$$B_n(x)=\sum_{k=0}^n\binom{n}{k}B_{n-k}x^k.$$
For $\mathbf a=(a_1,\ldots,a_r)$, the Bernoulli-Barnes numbers (see \cite{barnes}) are 
\begin{equation}\label{bja}
B_j(\mathbf a)=\sum_{i_1+\cdots+i_r=j} \binom{j}{i_1,\ldots,i_r}B_{i_1}\cdots B_{i_r}a_1^{i_1}\cdots a_r^{i_r}.
\end{equation}
From \eqref{ak} and \eqref{bja} it follows that
$$B_j(\mathbf a[k])=\sum_{i_1+\cdots+i_{rk}=j}\binom{j}{i_1,\ldots,i_{rk}}B_{i_1}\cdots B_{i_{rk}}a_1^{i_1+\cdots+i_k}
\cdots a_r^{i_{rk-k+1}+\cdots+i_{rk}} = $$ \small
\begin{equation}\label{bjak}
= \sum_{\ell_1+\cdots+\ell_r = j} \binom{j}{\ell_1,\ldots,\ell_r} a_1^{\ell_1}\cdots a_r^{\ell_r} 
  \sum_{\substack{i_1+\cdots+i_k=\ell_1 \\ \vdots \\ i_{rk-k+1}+\cdots+i_{rk}=\ell_r}} \binom{\ell_1}{i_1,\ldots,i_k}
		\cdots \binom{\ell_r}{i_{rk-k+1},\ldots,i_{rk}} B_{i_1}\cdots B_{i_{rk}}.
\end{equation} \normalsize
We consider the $rkD \times rkD$ determinant:
\begin{equation}\label{drcr}
\Delta(r,k,D):=\begin{vmatrix} 
\frac{B_1(\frac{1}{D})}{1} & \cdots & \frac{B_1(1)}{1} & \cdots & \frac{B_{rk}(\frac{1}{D})}{rk} & \cdots & \frac{B_{rk}(1)}{rk} \\
\frac{B_2(\frac{1}{D})}{2} & \cdots & \frac{B_1(1)}{1} & \cdots & \frac{B_{rk+1}(\frac{1}{D})}{rk+1} & \cdots & \frac{B_{rk+1}(1)}{rk+1}  \\
\vdots & \vdots & \vdots & \vdots & \vdots & \vdots & \vdots \\
\frac{B_{rkD}(\frac{1}{D})}{rkD} & \cdots & \frac{B_{rkD}(1)}{rkD} & \cdots & \frac{B_{rkD+rk-1}(\frac{1}{D})}{rkD+rk-1} & \cdots & 
\frac{B_{rkD+rk-1}(1)}{rkD+rk-1} \end{vmatrix}.
\end{equation}

\begin{prop}\label{dety}
If $\Delta(r,k,D)\neq 0$ then $\pak(n)$ can be expressed in terms of $B_j\left( \frac{v}{D} \right)$
where $1\leq v\leq D$ and $1\leq j\leq rkD+rk-1$, and $B_j(\mathbf a[k])$ with $rk \leq j\leq rkD+rk-1$.
\end{prop}

\begin{proof}
According to \cite[(1.8)]{medi}, we have that
\begin{equation}\label{crux}
\sum_{m=0}^{rk-1}\sum_{v=1}^{D} d_{k,\mathbf a,m}(n) D^{n+m} \frac{B_{n+m+1}\left(\frac{v}{D}\right)}{n+m+1} = 
  \frac{(-1)^{rk}n!}{(n+rk)!}B_{rk+n}(\mathbf a[k]) -\delta_{0n},
\end{equation}
where $\delta_{0n}=\begin{cases} 1,& n=0 \\ 0,& n>0 \end{cases}$.

Taking $n=0,1,\ldots,rkD-1$ in \eqref{crux} and seing $d_{k,\mathbf a,m}(n)$'s as unknowns, we obtain a linear system of type $rkD\times rkD$, whose determinant is $\Delta(r,k,D)$. Therefore, if $\Delta(r,k,D)\neq 0$, then $d_{k,\mathbf a,m}(n)$'s are the solutions of the above system.
Since, by Proposition \ref{p3}, we have 
$$\pak(n)=d_{k,\mathbf a,rk-1}(n)n^{rk-1}+\cdots+d_{k,\mathbf a,1}(n)n+d_{k,\mathbf a,0}(n),$$ 
we get the required result.
\end{proof}

We end this section with the following nice corollary of a result from \cite{graj}.

\begin{cor}\label{grajd}
If $m>1$ is a positive integer, then
$$ \lim_{N\to\infty} \frac{\#\{n\leq N\;:\;\pak(n)\not\equiv 0(\bmod\; m)\}}{N} \geq \frac{1}{k \sum\limits_{i=1}^r a_i}.$$
\end{cor}

\begin{proof}
It follows from the fact that $\pak(n)=p_{\mathbf a[k]}(n)$ (see \eqref{ak}) and \cite[Theorem 5.2]{graj}.
\end{proof}

\section{The polynomial part and the waves of $\pak(n)$}

Let $\mathbf a=(a_1,\ldots,a_r)$.
Sylvester \cite{sylvester,sylvesterc,sylv} decomposed the restricted partition function $\paa(n)$ in a sum of "waves", 
\begin{equation}\label{wave}
\paa(n)=\sum_{j\geq 1} W_{j}(n,\mathbf a), 
\end{equation}
where the sum is taken over all distinct divisors $j$ of the components of $\mathbf a$ and showed that for each such $j$, 
$W_j(n,\mathbf a)$ is the coefficient of $t^{-1}$ in
$$ \sum_{0 \leq \nu <j,\; \gcd(\nu,j)=1 } \frac{\left(\frac{2\pi \nu a_1 i}{j}\right)^{-\nu n}\cdot e^{nt}}
{(1- e^{-a_1t + \frac{2\pi \nu a_1 i}{j}})\cdots (1- e^{-a_rt + \frac{2\pi \nu a_r i}{j}})},$$
where $\gcd(0,0)=1$ by convention. Note that $W_{j}(n,\mathbf a)$'s are quasi-polynomials of period $j$.
Also, $W_1(n,\mathbf a)$ is called the \emph{polynomial part} of $\paa(n)$ and it is denoted by $\Pa(n)$.
We define:
\begin{equation}
W_j(n,k,\mathbf a):=W_j(n,\mathbf a[k])\text{ for }j\geq 1\text{ and }\Pak(n):=W_1(n,k,\mathbf a),
\end{equation}
the waves, respectively the polynomial part, of $\pak(n)$.

\begin{teor}\label{teo3}
For any positive integer $j$ with $j|a_i$ for some $1\leq i\leq r$, we have that:
\begin{align*}
& W_j(n,k,\mathbf a) = \frac{1}{D(rk-1)!}\sum_{m=1}^{rk-1}\sum_{\ell=1}^j e^{\frac{2\pi \ell i}{j}} \sum_{t=m-1}^{rk-1} \stir{rk}{t+1} \binom{t}{m-1} \times \\
& \times \sum_{\substack{(\ell_1,\ldots,\ell_r)\in\mathbf C \\ a_1\ell_1+\cdots+a_r\ell_r \equiv n(\bmod\;D)}} 
\prod_{s=1}^r \sum_{i_s,j_s\geq 0,\;i_s\frac{D}{a_s}+j_s=\ell_s}(-1)^{i_s}\binom{k}{i_s}\binom{j_s+k-1}{j_s} \times \\
& \times D^{-k}(a_1\ell_1+\cdots+a_r\ell_r)^{t-m+1}n^{m-1},
\end{align*}
where $\mathbf C=\{(\ell_1,\ldots,\ell_r)\;:\;0\leq \ell_1 \leq k(\frac{D}{a_1}-1),\ldots,0\leq \ell_r\leq k(\frac{D}{a_r}-1)\}$.
\end{teor}

\begin{proof}
The proof is similar to the proof of Theorem \ref{teo1}, using \cite[Proposition 4.2]{remarks}.
\end{proof}

\begin{teor}\label{teo4}
For $n\geq 0$ we have that
\begin{align*}
& \Pak(n) = \frac{1}{(rk-1)!} \sum_{(\ell_1,\ldots,\ell_r)\in\mathbf C} 
\prod_{s=1}^r \sum_{i_s,j_s\geq 0,\;i_s\frac{D}{a_s}+j_s=\ell_s}(-1)^{i_s}\binom{k}{i_s}\binom{j_s+k-1}{j_s} \times \\
& \times \prod_{t=1}^{rk-1}\left( \frac{n-a_1\ell_1-\cdots-a_r\ell_r}{D} + t \right),
\end{align*}
where $\mathbf C=\{(\ell_1,\ldots,\ell_r)\;:\;0\leq \ell_1 \leq k(\frac{D}{a_1}-1),\ldots,0\leq \ell_r\leq k(\frac{D}{a_r}-1)\}$.
\end{teor}

\begin{proof}
The proof is similar to the proof of Theorem \ref{teo1}, using \cite[Corollary 3.6]{lucrare}.
\end{proof}

\begin{teor}\label{teo5}
We have 
\begin{align*}
& \Pak(n)= \frac{1}{(a_1\cdots a_r)^k}\sum_{u=0}^{rk-1}\frac{(-1)^u}{(rk-1-u)!} n^{rk-1-u} 
\sum_{\ell_1+\cdots+\ell_r=u} a_1^{\ell_1}\cdots a_r^{\ell_r} \sum_{\substack{i_1+\cdots+i_k=\ell_1 \\ \vdots \\ i_{rk+k-1}+\cdots+i_{rk}=\ell_r}}
  \frac{B_{i_1}\cdots B_{i_{rk}}}{i_1!\cdots i_{rk}!} 
\end{align*}
\end{teor}

\begin{proof}
From \cite[Corollary 3.11]{lucrare} it follows that
$$\Pak(n)= \frac{1}{(a_1\cdots a_r)^k}\sum_{u=0}^{rk-1}\frac{(-1)^u}{(rk-1-u)!}\sum_{i_1+\cdots+i_{rk}=u} 
\frac{B_{i_1}\cdots B_{i_{rk}}}{i_1!\cdots i_{rk}!}a_1^{i_1+\cdots+i_k}\cdots a_r^{i_{rk-k+1}\cdots i_{rk}} n^{rk-1-u} $$
The conclusion follows immediately.
\end{proof}


\section{Conclusions}

We proved new formulas for $\pak(n)$, the restricted $k$-multipartition function associated to a sequence of 
positive integers $\mathbf a =(a_1,\ldots,a_r)$ and to an integer $k\geq 2$, its Sylvester's waves and, in particular, its polynomial part.
Also, we give a lower bound for the density of the set $\{n\geq 0\;:\;\pak(n)\not\equiv 0(\bmod\;m)\}$, where $m\geq 2$.

Our methods are suitable for study other (restricted) integer partition functions.

\subsection*{Acknowledgments} 

The first author (Mircea Cimpoea\c s) was supported by a grant of the Ministry of Research, Innovation and Digitization, CNCS - UEFISCDI, 
project number PN-III-P1-1.1-TE-2021-1633, within PNCDI III.






\end{document}